\documentclass{article} 
\usepackage{times}  
\usepackage{helvet}  
\usepackage{courier}  
\usepackage[hyphens]{url}
\usepackage{graphicx} 
\usepackage{subfigure}
\usepackage{natbib}
\usepackage{caption}

\usepackage[utf8]{inputenc}
\usepackage{tikz}
\usepackage{amsmath, amsthm, mathrsfs, mathtools, amsfonts}
\usepackage{nicefrac}
\usepackage{MnSymbol}
\usepackage{algorithm}
\usepackage{algorithmicx}
\usepackage{algpseudocode}
    \algrenewcommand\algorithmicindent{1em}
\usepackage{xcolor}
\usepackage{aligned-overset} 
\usepackage{authblk}
\usepackage{fullpage}
\usepackage[capitalize,noabbrev]{cleveref} 

\newcommand{\cI}{\mathcal{I}}
\newcommand{\cJ}{\mathcal{J}}

\newcommand{\cP}{\mathcal{P}}

\newcommand{\I}{\mathbb{I}}

\newcommand{\N}{\mathbb{N}}

\newcommand{\bbR}{\mathbb{R}}

\newcommand{\tht}{\vartheta}
\newcommand{\lrb}[1]{\left(#1\right)}
\newcommand{\brb}[1]{\bigl(#1\bigr)}

\newcommand{\bsb}[1]{\bigl[#1\bigr]}

\newcommand{\bcb}[1]{\bigl\{#1\bigr\}}

\newcommand{\labs}[1]{\left\lvert#1\right\rvert}

\DeclareMathOperator*{\argmin}{argmin}
\DeclareMathOperator*{\argmax}{argmax}

\newcommand{\lip}{Lipschitz}

\newcommand{\fracc}[2]{#1/#2}
\newcommand{\s}{\subset}
\newcommand{\m}{\setminus}
\newcommand{\iop}{\infty}
\newcommand{\nhphantom}[1]{\sbox0{#1}\hspace{-\the\wd0}}
\newcommand{\mypapertitle}{A Near-Optimal Algorithm for\\ Univariate Zeroth-Order Budget Convex Optimization}
\newcommand{\fB}{\mathfrak{B}}

\newcommand{\cancL}{\text{$\filledmedsquare${\hspace{-1pt}}$\medsquare${\hspace{-1pt}}$\square${\hspace{-1pt}}$\medsquare$}}
\newcommand{\cancR}{\text{$\medsquare${\hspace{-1pt}}$\medsquare${\hspace{-1pt}}$\medsquare${\hspace{-1pt}}$\filledmedsquare$}}
\newcommand{\cancLL}{\text{$\filledmedsquare${\hspace{-1pt}}$\filledmedsquare${\hspace{-1pt}}$\medsquare${\hspace{-1pt}}$\medsquare$}}
\newcommand{\cancRR}{\text{$\medsquare${\hspace{-1pt}}$\medsquare${\hspace{-1pt}}$\filledmedsquare${\hspace{-1pt}}$\filledmedsquare$}}
\newcommand{\cancLR}{\text{$\filledmedsquare${\hspace{-1pt}}$\medsquare${\hspace{-1pt}}$\medsquare${\hspace{-1pt}}$\filledmedsquare$}}
\newcommand{\cancNone}{\text{$\medsquare${\hspace{-1pt}}$\medsquare${\hspace{-1pt}}$\medsquare${\hspace{-1pt}}$\medsquare$}}
\newcommand{\algName}{Dyadic Search}
\newcommand{\fupd}{\mathrm{update}}
\newcommand{\fdel}{\mathrm{delete}}

\newcommand{\unif}{\textcolor{black!50!gray}{\mathfrak{u}}}
\newcommand{\nunif}{\textcolor{black!50!gray}{\not{\mathfrak{u}}}}
\newcommand{\del}{\mathrm{del}}

\newtheorem{theorem}{Theorem}
\newtheorem{assumption}{Assumption}%

\title{\mypapertitle{}}

\author[1]{Fran\c{c}ois Bachoc}
\author[2,3]{Tommaso Cesari}
\author[2,4]{Roberto Colomboni}
\author[2,4]{Andrea Paudice}

\affil[1]{Institut de Math\'ematiques de Toulouse, Toulouse, France}
\affil[2]{Universit\`a degli Studi di Milano, Milano, Italy}
\affil[3]{Toulouse School of Economics, Toulouse, France}
\affil[4]{Istituto Italiano di Tecnologia, Genova, Italy}

\begin{document}

\maketitle

\begin{abstract}
This paper studies a natural generalization of the problem of minimizing a univariate convex function $f$ by querying its values sequentially. At each time-step $t$, the optimizer can invest a budget $b_t$ in a query point $X_t$ of their choice to obtain a fuzzy evaluation of $f$ at $X_t$ whose accuracy depends on the amount of budget invested in $X_t$ across times. This setting is motivated by the minimization of objectives whose values can only be determined approximately through lengthy or expensive computations. We design an any-time parameter-free algorithm called Dyadic Search, for which we prove near-optimal optimization error guarantees. As a byproduct of our analysis, we show that the classical dependence on the global Lipschitz constant in the error bounds is an artifact of the granularity of the budget. Finally, we illustrate our theoretical findings with numerical simulations.
\end{abstract}

\section{Introduction}\label{s:intro}
In this paper we consider a variant of the following fundamental question: given oracle access to the values of a convex real-valued function $f$, what is the best approximation in value to the infimum of $f$ (over an interval $I$) that is possible to achieve given a certain number of oracle queries?

This problem is known as zeroth order convex optimization and has been studied from at least \cite{rosenbrock1960}. The field has also recently attracted interest in the machine learning and statistical community due to the fact that computing the gradient of a function that depends on a large dataset, e.g. the empirical risk, can be very expensive if not unfeasible; see for example \citet{bubeck2021} and references therein. Another important application arises in simulation-based optimization, where the goal is to optimally tune the parameters of a system, by only observing its output \cite{conn2009, spall2005}.

In the basic deterministic setting at each query $x$ the oracle answers with $f(x)$. This type of oracle feedback can be restrictive since it does not cover the case where the value of $f$ can only be estimated. The stochastic setting, where the oracle answers with a noisy realization of $f(x)$, overcomes this limitation. One limitation of the stochastic setting is that oracle answers are usually assumed to be conditionally independent from the past. This is natural in applications where the objective is defined by an unknown distribution that can only be accessed via sampling, such as those appearing in statistics. However, this setting doesn't allow for an appropriate modelling of applications where the function values at a query point $x$ can be computed with incremental precision. This is the case of functions that are expressed through infinite series, or whose values is computed through lengthy simulations at a precision that is proportional to the running time that the simulation algorithm has spent on that input.

To overcome these limitations, we propose a novel setting we name \emph{zeroth-order budget convex optimization}. In this setting, at each query $x$ the oracle answers with an interval which is guaranteed to contain $f(x)$ and whose length decreases with the \emph{budget} invested on $x$ so far. This generalizes the deterministic and the stochastic settings and additionally allows for modelling the reuse of information required by the applications described above.

We focus on the one dimensional case which enjoys the rich structures of the reals and allows us to derive bounds that feature fine-grained dependencies on the problem parameters. In particular, we give the following contributions:

\begin{itemize}
\item We design a novel zeroth-order optimization setting that allows to model situations where the objective values can be computed with incremental precision; to the best of our knowledge, this is the first setting that explicitly model the reuse of the information on the issued queries.
\item We design a near-optimal algorithm, Dyadic Search, that works under a minimal convexity assumption on $f$. The algorithm is fully practicable as it works \emph{anytime} and is \emph{parameter free}.
\item We prove a sharp \emph{anytime} upper bound on the optimization error of Dyadic Search that is the sum of two terms: the first one decreases polynomially with the budged, and does not depend on \emph{any} \lip{} constant; the second one depends only on a \emph{local} \lip{} constant and decreases exponentially with the budget (in particular, this yields that the dependence on the \lip{} constant is asymptotically negligible).
\item We prove a matching (up to constants) lower bound. The lower bound holds in the more restrictive finite-horizon setting and then certifies the optimality of Dyadic Search even under stronger assumptions.
\item We show that the dependence on a global or local \lip{} constant of $f$, that (to the best of our knowledge) classically appears in all results ever stated, is an artifact of the discrete nature of convex optimization problems, and can be entirely lifted transitioning to a continuous budget optimization setting (this claim will be expressed more precisely in the discussion following \Cref{t:upper}).
\end{itemize}

\paragraph{Related Work.} 
Zeroth-order optimization is a huge field with a vast literature. Aligned with the scope of this paper, we limit our discussion to the one-dimensional case. 

The deterministic case is the simplest setting in zeroth-order optimization, where at each query $x$ the oracle answers with the exact value of the objective $f(x)$. When $f$ is $L$-Lipschitz, any simple bisecting strategy achieves an optimization error of $\mathcal{O}(L|I|\exp{\left(-t\right)})$ after $t$ queries; this performance matches the lower bound in \citealt[Theorem~3.2.8]{nesterov2018} and then is optimal (up to constant factors). 
Our method, recovers the optimal performance of these bisection strategies when the oracle answers with $f(x)$ each query, moreover it does not require the objective to be globally Lipschitz. 

In stochastic zeroth-order optimization, the oracle answers each query with one (or more) random noisy estimates of the objective, thus generalizing the deterministic case. In this setting the state-of-the-art is more complex and the best achievable performance (and then the related algorithms) depends on the underlying noise and oracle model. 
One of the most studied (and possibly natural) setting is when the oracle answers each query $x$ with a single random noisy estimate $\hat{f}(x; \xi) = f(x) + \xi$, where $\xi$ is a conditionally independent (from the past queries) $0$ mean sub-Gaussian random variable. 
In this setting, the best upper bound on the expected optimization error for convex functions, when specialized to the one-dimensional case, is $\mathcal{O}(\nicefrac{\log(t|I|)}{\sqrt{t}})$ after $t$ queries (see \citealt[Theorem~1]{Lattimore2020}); this method requires the knowledge of the time-horizon. An optimization error of $\Omega(\nicefrac{1}{\sqrt{t}})$ is unavoidable, and then the method proposed in \citealt{Lattimore2020} is nearly optimal, even knowing the time-horizon in advance and under the additional assumptions of smoothness and strong-convexity (see \citealt[Theorem~3]{shamir2013}). The dependence on $\nicefrac{1}{\sqrt{t}}$ is a consequence of the sub-Gaussianity of the confidence intervals. We remark that, in our setting, assuming that the confidence intervals provided by the oracle have the same sub-Gaussian shrinking speed, our method obtains the optimal optimization error of $\mathcal{O}(\nicefrac{1}{\sqrt{t}})$ without requiring the time-horizon in advance.

The nearly optimal bound of \citealt{Lattimore2020} is obtained by a quite involved algorithm. A slightly worse bound of $\mathcal{O} \brb{ (\nicefrac{\sqrt{\log t}}{\sqrt{t}}) \log(\nicefrac{t}{\log t} }$ is achieved with high probability in \citealt[Theorem~1]{Agarwal2013} with a much simpler algorithm. Our method, Dyadic Search, shares important similarities with the trisection method analyzed in \citealt{Agarwal2013}: both the methods monitor the confidence intervals separation around the active points to establish which portion to the domain to discard. Our analysis, however, departs substantially from that of \citealt{Agarwal2013}. Beyond obtaining a tighter bound of $\mathcal{O}(\nicefrac{1}{\sqrt{t}})$, we leverage the incremental structure of the budgeted setting to crave out a fine dependence on the local Lipschitz constant in the optimization error bound. This gives a bound that, depending on the problem instance, can be much sharper than the ones featuring the global Lipschitz constant; besides, we can optimize functions that are not even Lipschitz. We are also able to design a recommendation strategy that makes our algorithm any-time as opposed to that in \citealt{Agarwal2013}.

\paragraph{Outline of the paper.} 
The rest of the paper is organized as follows. 
In \Cref{s:setting} we define the zeroth-order budget convex optimization problem and the required notation. 
In \Cref{s:algo,s:upper} we describe Dyadic Search and its analysis, respectively. 
In \Cref{s:lower} we state and prove an information-theoretic lower bound. 
In \Cref{s:experiments} we present numerical simulations that validate our theory. 
Finally, in \Cref{s:conclusions} we draw the conclusions.

\section{Setting}\label{s:setting}

In this section, we introduce the formal setting for our budget convex optimization problem.

Given a bounded interval $I\s \bbR$, our goal is to minimize an unknown \emph{convex} function $f\colon I \to \bbR$ picked by a possibly adversarial and adaptive environment by only requesting fuzzy evaluations of $f$.
At every interaction $t$, the optimizer is given a certain budget $b_t$ that can be invested in a query point $X_t$ of their choosing to reduce the fuzziness of the value of $f(X_t)$, modeled by an interval $J_t \ni f(X_t)$.

The interactions between the optimizer and the environment are described in Optimization~Protocol~\ref{a:protocol}.

{
\makeatletter
\renewcommand{\ALG@name}{Optimization Protocol}
\makeatother

\begin{algorithm}
\caption{\label{a:protocol}}
\textbf{input:} A non-empty bounded interval $I\s\bbR$ (the domain of the unknown objective $f$)
\begin{algorithmic}[1]
\For{$t=1,2,\dots$}
    \State The environment picks and reveals a budget $b_t > 0$
    \State The optimizer selects a query point $X_t \in I$ where to invest the budget $b_t$
    \State The environment picks and reveals an interval $J_t \s \bbR$ such that $f(X_t) \in J_t$
    \State The optimizer recommends a point $R_t \in I$
\EndFor 
\end{algorithmic}
\end{algorithm}
}

We stress that the environment is adaptive. 
Indeed, the intervals $J_t$ that are given as answers to the queries $X_t$ can be chosen by the environment as an arbitrary function of the past history, as long as they represent fuzzy evaluations of the convex function $f$, i.e., $f(X_t) \in J_t$.

Note that optimization would be impossible without further restrictions on the behavior of the environment, since an adversarial convex environment could return $J_t=\mathbb R$ for all $t\in \N$, making it impossible to gather any meaningful information.
We limit the power of the environment by relating the amount of budget invested in a query point $X_t$ with the length of the corresponding fuzzy representation $J_t$ of $f(X_t)$.
The idea is that the more budget is invested, the more accurate approximation of the objective $f$ can be determined, in a quantifiable way.
This is made formal by the following assumption.
\begin{assumption}
\label{ass:budget}
There exist $c \ge 0$ and $\alpha > 0$ such that, for any $t \in \N$, if the optimizer invested the budgets $b_1, \dots, b_t$ in the query points $X_1, \dots, X_t$, then
\[
    \labs{ J_t }
\le
    \frac{c}{\fB_t^\alpha} \;,
\]
where $\labs{J_t}$ denotes the length of $J_t$ and $\fB_t \coloneqq \sum_{s=1}^t b_s\I\{X_s=X_t\}$ is the total budget invested in $X_t$ up to time $t$.
\end{assumption}

The performance after $T$ interactions of an algorithm that received budgets $b_1,\dots,b_T$ is evaluated with the optimization error of the recommendation $R_T$.
More precisely, we want to control the difference  
$
    f(R_T) - \inf_{x\in I}f(x)
$, 
for any choice of the convex function $f$ and the fuzzy evaluations $J_1,\dots,J_T$.

\section{\algName}\label{s:algo}
In this section, we present our \algName{} algorithm for budget convex optimization (\Cref{a:dyadic}).

Before presenting its pseudo-code, we introduce some notation.
For any positive integer $n\in \N$ we denote by $[n]$ the set $\{1,\dots, n\}$ of the first $n$ integers.
Let $\cP \coloneqq \{ \cancL, \cancR, \cancLL, \cancRR, \cancLR \}$.
The blackened parts of the elements of $\cP$ represent which portions of the active interval maintained by \algName{} the algorithm will delete.
Additionally, we will consider the element $\cancNone$ representing the case where no parts of the active interval will be deleted.
Let $\cJ$ be the set of all intervals, and $\cI \s \cJ$ that of all \emph{bounded} intervals.
Furthermore, for any interval $J\in \cJ$, let
\[
    J^- \coloneqq \inf (J)
    \qquad \text{and} \qquad
    J^+ \coloneqq \sup (J) \;.
\]
\algName{} relies on four auxiliary functions: the $\fdel$ function, the uniform partition function $\unif$, the non-uniform partition function $\nunif$, and the $\fupd$ function.
The $\fdel$ function (see \Cref{f:delete-function})
\[
    \fdel \, \colon \, \cJ^3 \to \cP \cup \{ \cancNone \}
\]
is defined, for all $(J_l, J_{c}, J_{r}) \in \cJ^3$, by
\[
    \begin{cases}
    \cancLL & \text{if } J_c^- \ge J_r^+\text{, else}\\
    \cancRR & \text{if } J_c^- \ge J_l^+\text{, else}\\
    \cancLR & \text{if } J_l^- \ge \min ( J_c^+, J_r^+ ) \text{ and } J_r^- \ge \min ( J_l^+, J_c^+)\text{, else}\\
    \cancL & \text{if } J_l^- \ge \min ( J_c^+, J_r^+ )\text{, else}\\
    \cancR & \text{if } J_r^- \ge \min ( J_l^+, J_c^+ )\text{, else}\\
    \cancNone  & .
    \end{cases}
\]
\begin{figure}
    \centering
    \begin{tikzpicture}[scale = 0.35]
    \draw[->] (-1,0) -- (21,0);
    \draw[very thick] 
        (0,-0.3) -- (0,0.3)
        (5,-0.3) -- (5,0.3)
        (10,-0.3) -- (10,0.3)
        (15,-0.3) -- (15,0.3)
        (20,-0.3) -- (20,0.3)
    ;
    \draw 
        (5,0) node[below right] {$l$}
        (10,0) node[below right] {$c$}
        (15,0) node[below left] {$r$}
        (0,2.5) node[left] {$J_l^+$}
        (0,4) node[left] {$J_r^-$}
        (5,2.5) node[right] {$J_l$}
        (10,8) node[left] {$J_c$}
        (15,7) node[left] {$J_r$}
    ;
    \draw[dashed, gray] 
        (0,2.5) -- (5, 2.5)
        (0,4) -- (15, 4)
    ;
    \draw[thick, blue] 
        (20,9) -- (9,0) parabola (0,5.5) node[blue, above left] {$f$};
    \draw[thick, gray]  
        (5,-1) -- (5,2.5)
        (10,-0.6) -- (10,8)
        (15,4) -- (15,7)
        ;
    \begin{scope} 
        \clip (15.1,-1) rectangle (19.9, 9);
        \foreach \x in {-35,...,20}
        {
            \draw[line width=2pt, darkgray] (10, {20 + \x/2}) -- (30, {\x/2});
        }
        \end{scope}
    \end{tikzpicture}
    \caption{A representation of the $\fdel$ function. Since $J_l^+ \le J_r^-$, the points right of $r$ are deleted, i.e., $\fdel(J_l,J_c,J_r) = \cancR$.}
    \label{f:delete-function}
\end{figure}
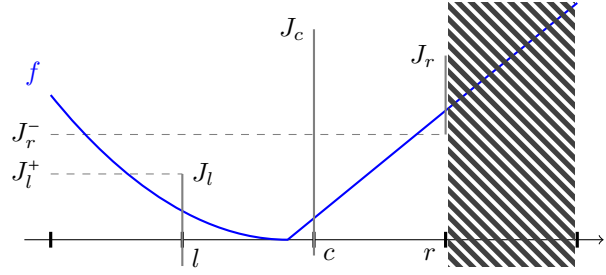
In words, the intervals $J_l, J_c, J_r$ will represent the fuzzy evaluations of three points $l < c < r$ in the domain of the unknown objective (left, center, and right).
Since we are assuming that the objective is convex (hence unimodal\footnote{By unimodal, we mean that there exists a point $x$ belonging to the closure of the domain of $f$ such that $f$ is nonincreasing before $x$ and nondecreasing after $x$. More precisely, either $f$ is nonincreasing on the domain intersected with $(-\iop,x]$ and nondecreasing on the domain intersected with $(x,\iop)$ or is nonincreasing on the domain intersected with $(-\iop,x)$ and nondecreasing on the domain intersected with $[x,\iop)$.}), note that whenever an upper bound on the value of the objective at a point $x$ is lower than the lower bound at another point $y$ that is left (resp., right) of $x$, then, all points that are left (resp., right) of $y$ ($y$ included) are no better than $x$.
Therefore, the function $\fdel$ returns which part of an interval containing three distinct points $l < c < r$ should be deleted given the fuzzy evaluations $J_l, J_c, J_r$. (E.g., $\cancLL$ represents the deletion of all points of the active interval left of $c$,
$\cancR$ represents the deletion of all points of the active interval right of $r$,
$\cancNone$ is returned when the fuzzy evaluations are not sufficient to delete anything, etc.)

The uniform and non-uniform partition functions are defined, respectively, by
\begin{align*}
    \unif \, \colon \, \cI & \textstyle{ \to \bbR^3 \;, \quad I \mapsto \brb{ \frac34 I^- + \frac14 I^+, \, \frac12 I^- + \frac12 I^+, \, \frac14 I^- + \frac34 I^+ } } \;, \\
    \nunif \, \colon \, \cI & \textstyle{ \to \bbR^3 \;, \quad I \mapsto \brb{ \frac23 I^- + \frac13 I^+, \, \frac12 I^- + \frac12 I^+, \, \frac13 I^- + \frac23 I^+ } } \;.
\end{align*}
In words, when applied to an interval $I$, the uniform partition function $\unif$ returns the three points that are at $\nicefrac14$,  $\nicefrac12$, and $\nicefrac34$ of the interval, while the non-uniform partition function $\nunif$ returns the three points that are at $\nicefrac13$, $\nicefrac12$, and $\nicefrac23$ of the interval (see \Cref{f:unif-nunif}).
\begin{figure}
    \centering
    \begin{tikzpicture}[scale=3]
    \draw (0,0) -- (1,0);
    \foreach \x in {0,1}
    {
        \draw (\x, -0.05) -- (\x, 0.05);
    }
    \foreach \x in {1/4,1/2,3/4}
    {
        \draw[blue] (\x, -0.05) -- (\x, 0.05);
    }
    \draw 
        (0,-0.05) node[below] {$0$}
        (1/4,-0.05) node[below, blue] {$\nicefrac14$}
        (1/2,-0.05) node[below, blue] {$\nicefrac12$}
        (3/4,-0.05) node[below, blue] {$\nicefrac34$}
        (1,-0.05) node[below] {$1$}
    ;
    \draw (1/2,0.3) node {$\unif$};
    \end{tikzpicture}
    \qquad
    \begin{tikzpicture}[scale=3]
    \draw (0,0) -- (1,0);
    \foreach \x in {0,1}
    {
        \draw (\x, -0.05) -- (\x, 0.05);
    }
    \foreach \x in {1/3,1/2,2/3}
    {
        \draw[blue] (\x, -0.05) -- (\x, 0.05);
    }
    \draw 
        (0,-0.05) node[below] {$0$}
        (1/3,-0.05) node[below, blue] {$\nicefrac13$}
        (1/2,-0.05) node[below, blue] {$\nicefrac12$}
        (2/3,-0.05) node[below, blue] {$\nicefrac23$}
        (1,-0.05) node[below] {$1$}
    ;
    \draw (1/2,0.3) node {$\nunif$};
    \end{tikzpicture}
    \caption{The uniform ($\protect\unif$) and non-uniform ($\protect\nunif$) partition functions applied to the interval $I=[0,1]$.}
    \label{f:unif-nunif}
\end{figure}
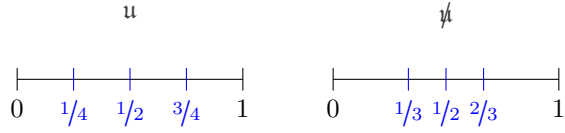

The $\fupd$ function
\[
    \fupd \, \colon \, \cI \times \{ \unif, \nunif \} \times \cP \to \cI \times \{ \unif, \nunif \}
\]
is defined, for all $(I,\tht,\del) \in  \cI \times \{ \unif, \nunif \} \times \cP$, by the following table:
\[
\begin{matrix}
& \unif & \nunif \\
\cancLL & \brb{ \bsb{ \frac12I^-+\frac12I^+, \, I^+ } , \, \unif  } & \brb{ \bsb{ \frac12I^-+\frac12I^+, \, I^+ } , \, \nunif  } \\
\cancRR & \brb{ \bsb{ I^-, \, \frac12I^-+\frac12I^+ } , \, \unif  } & \brb{ \bsb{ I^-, \, \frac12I^-+\frac12I^+ } , \, \nunif  } \\
\cancLR & \brb{ \bsb{ \frac{3I^- + I^+}{4}, \, \frac{I^- + 3I^+}{4} }, \, \unif  } & \brb{ \bsb{ \frac{2I^- + I^+}{3}, \, \frac{I^- + 2I^+ }{3} } , \, \unif  }\\
\cancL & \brb{ \bsb{ \frac34I^-+\frac14I^+, \, I^+ } , \, \nunif  } & \brb{ \bsb{ \frac23I^-+\frac13I^+, \, I^+ } , \, \unif  } \\
\cancR & \brb{ \bsb{ I^-, \, \frac14I^-+\frac34I^+ } , \, \nunif  } & \brb{ \bsb{ I^-, \, \frac13I^-+\frac23I^+ } , \, \unif  } \\
\end{matrix}
\]
In words, when applied to an interval $I$, a type of partition $\tht$, and the subset of $I$ to be deleted modeled by $\del$, the $\fupd$ function returns as the first component the interval $I$ pruned of the subset of $I$ specified by $\tht$ and $\del$, and, as the second component, how the new interval will be partitioned (see \Cref{f:dyadic}).
It can be seen that the types of partitions returned by $\fupd$ are chosen so that our \algName{} algorithms will only query points on a (rescaled) dyadic mesh. (E.g., if $I=[0,1]$, \algName{} will only query points of the form $k/2^{h}$, for $k,h\in \N$.)

For all $t \in \N$, if the sequence of budgets picked by the environment up to time $t$ is $b_1,\dots,b_t$ and the sequence of query points selected by the optimizer is $X_1,\dots,X_t$, for each $x\in \bbR$, we define the quantities
\[
    \fB_{x,t}
\coloneqq
    \sum_{s=1}^t b_s \I \{X_s = x\}
\qquad
\text{and}
\qquad
    J_{x,t}
\coloneqq
    \bigcap_{s\in[t], X_s = x} J_s
\]
with the understanding that $J_{x,t}=\bbR$ whenever $X_s \neq x$ for all $s \in [t]$.
Furthermore, define $\fB_{x,0} = 0$ for all $x\in \bbR$.
In words, $\fB_{x,t}$ is the total budget that has been invested in $x$ by the optimizer up to and including time $t$, while $J_{x,t}$ is the best fuzzy evaluation of the unknown objective at $x$ that is available at the end of time $t$.

The pseudocode of \algName{} is provided in \Cref{a:dyadic}.
\algName{} proceeds in epochs $\tau$ where it maintains an active interval $I_\tau$ and three query points $l_{\tau}, c_{\tau}, r_{\tau} \in I_\tau$.
During each epoch $\tau$, it repeatedly queries a point in $\{ l_{\tau}, c_{\tau}, r_{\tau} \}$ where it invested the least amount of budget (\Cref{s:query}) until the function $\fdel$ has gathered enough information to prune the current active interval (\Cref{state:recommendation-one}).
When this happens, first it updates the active interval and the type of partition using the $\fupd$ function (\Cref{state:update}).
Then, it computes the three query points $l_{\tau+1}, c_{\tau+1}, r_{\tau+1}$ of the next epoch $\tau+1$ (\Cref{s:query-points}). Notably, among $l_{\tau+1}, c_{\tau+1}, r_{\tau+1}$ there will be the point among $l_{\tau}, c_{\tau}, r_{\tau}$ that has the smallest value of $f$ (or one of them, if there are more than one).
Afterwards, the algorithm recommends a point $x \in \{l_{\tau+1}, c_{\tau+1}, r_{\tau+1}\}$ with the best known upper bound $J^+_{x,t}$ on the value of $f(x)$ available at the present time $t$,\footnote{Under \Cref{ass:budget}, this corresponds to recommending a point $x \in \{l_{\tau}, c_{\tau}, r_{\tau}\}$ with the best known upper bound $J^+_{x,t}$ on the value of $f(x)$ that will ``survive'' as a query point of the next epoch. Indeed, for $x \in \{l_{\tau+1}, c_{\tau+1}, r_{\tau+1}\} \m \{l_{\tau}, c_{\tau}, r_{\tau}\}$, we have $J^+_{x,t}=+\iop$, since $x$ has never been evaluated. On the other hand, any $x \in \{l_{\tau}, c_{\tau}, r_{\tau}\}$ has already been evaluated, hence $J^+_{x,t} < \iop$.} and concludes the current epoch (\Cref{state:recommendationone}). 
In all rounds in which function $\fdel$ has not yet gathered enough information to prune the current active interval, the algorithm makes different recommendations depending on whether or not the amount of budget invested in the current epoch is higher than the amount of budget spent in all past epochs combined
(\Cref{state:recommendationonebis,state:recommendationtwo} respectively). 
See \Cref{f:dyadic} for an illustration of how \algName{} works.
\begin{figure}
\centering
\begin{tikzpicture}
\def\ItauMinus{{
	0.0,
	0.0,
	0.0,
	0.0,
	0.0,
	0.0,
	0.0625,
	0.0625,
	0.0625,
	0.078125,
	0.078125,
	0.0859375
}}
\def\ItauPlus{{
	1.0,
	0.75,
	0.5,
	0.375,
	0.25,
	0.1875,
	0.1875,
	0.15625,
	0.125,
	0.125,
	0.109375,
	0.109375
}}
\def\queryl{{
    0.25,
    0.25,
    0.125,
    0.125,
    0.0625,
    0.0625,
    0.09375,
    0.09375,
    0.078125,
    0.09375,
    0.0859375,
    0.09375
}}
\def\queryc{{
    0.5,
    0.375,
    0.25,
    0.1875,
    0.125,
    0.09375,
    0.125,
    0.109375,
    0.09375,
    0.1015625,
    0.09375,
    0.09765625
}}
\def\queryr{{
    0.75,
    0.5,
    0.375,
    0.25,
    0.1875,
    0.125,
    0.15625,
    0.125,
    0.109375,
    0.109375,
    0.1015625,
    0.1015625
}}
\def\recommendation{{
    0.25,
    0.25,
    0.25,
    0.125,
    0.125,
    0.125,
    0.09375,
    0.09375,
    0.09375,
    0.09375,
    0.1015625,
    0.1015625
}}
\def\colorPalette{{
    
}}
\def\scalex{7}
\def\scaley{0.4}

\foreach \x in {0,1,...,11} {
    \pgfmathsetmacro{\myhue}{\x/13} 
    \definecolor{mycolor}{hsb}{\myhue,1,1} 
	\draw[mycolor] 
	    (\ItauMinus[\x]*\scalex,\x*\scaley) -- (\ItauPlus[\x]*\scalex, \x*\scaley);
	\draw[mycolor] 
	    (\queryl[\x]*\scalex,{(\x-0.3)*\scaley}) --
	    (\queryl[\x]*\scalex,{(\x+0.3)*\scaley})
	    (\queryc[\x]*\scalex,{(\x-0.3)*\scaley}) --
	    (\queryc[\x]*\scalex,{(\x+0.3)*\scaley})
	    (\queryr[\x]*\scalex,{(\x-0.3)*\scaley}) --
	    (\queryr[\x]*\scalex,{(\x+0.3)*\scaley});
}
\foreach \x in {0,1,...,11} {
    \pgfmathsetmacro{\myhue}{\x/13} 
    \definecolor{mycolor}{hsb}{\myhue,1,1} 
    \draw[mycolor]
        (\recommendation[\x]*\scalex,\x*\scaley) --
        (\recommendation[\x]*\scalex,{(-0.4)*\scaley});
}
\foreach \x in {0,1,...,11} {
    \pgfmathsetmacro{\myhue}{\x/13} 
    \definecolor{mycolor}{hsb}{\myhue,1,1} 
    \draw[mycolor, fill = white] 
        (\recommendation[\x]*\scalex,\x*\scaley) circle ({0.1*\scalex pt});
}
\draw[thick] 
    (0,0.1*\scalex*\scaley) -- 
    (0.1*\scalex,0) --
    (1*\scalex,0.9*\scalex*\scaley) node[above] {$f(x)$};
\draw[->]  (0,-1*\scaley) -- (1.025*\scalex,-1*\scaley) node[below] {$x$};
\draw[->]  (0,-1*\scaley) -- (0,12*\scaley) node[left] {$\tau$};
\end{tikzpicture}
\caption{A run of \algName{}.
Here, the function is piece-wise linear and its graph is in thick black. The horizontal segments are the active intervals $I_\tau$ of consecutive epochs $\tau$. The short vertical segments are the current query points $l_\tau, c_\tau, r_\tau$ of epoch $\tau$, and the dots (prolonged down vertically) are the recommendations at the end of each epoch, that converge towards $x^\star$.}
\label{f:dyadic}
\end{figure}
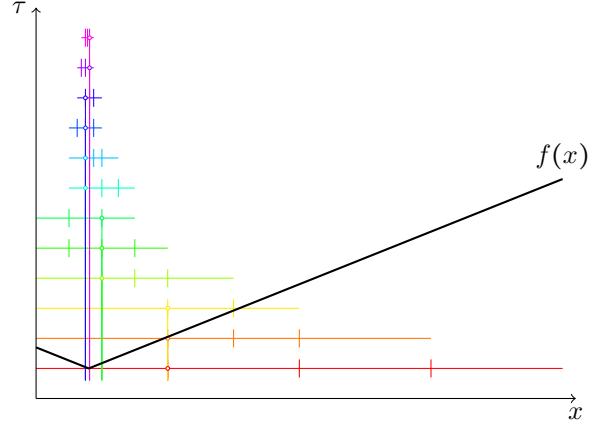

\begin{algorithm}
\caption{\label{a:dyadic}\algName}
\textbf{input:} A non-empty bounded interval $I\s\bbR$ (the domain of the unknown objective)

\textbf{initialization:} $I_1 \coloneqq [I^-,I^+]$, $\tht_1 \coloneqq \unif$, $(l_1,c_1,r_1) \coloneqq \tht_1(I_1)$, $t_0 \coloneqq 0$, $B_0 \coloneqq 0$, $B_{1,0} \coloneqq 0$
\begin{algorithmic}[1]
\For{epochs $\tau=1,2,\dots$}
    \For{$t=t_{\tau-1}+1,t_{\tau-1}+2,\dots$}
        \State Query $X_t \in \argmin_{x \in \{ l_\tau, c_\tau, r_\tau \} } \fB_{x,t-1}$ \label{s:query}
        \State Let $\del_t \coloneqq \fdel( J_{l_\tau,t}, J_{c_\tau,t}, J_{r_\tau,t} )$
        \State Let $B_{\tau,t} \coloneqq B_{\tau,t-1}+b_t$ (and $\tau_t \coloneqq \tau$) \label{state:extra-command}
        \If{$\del_t \neq \cancNone$\label{state:recommendation-one}}
            \State Let $t_\tau \coloneqq t$, $B_\tau \coloneqq B_{\tau,t}$, and $B_{\tau+1,t} \coloneqq 0$
            \State Let $(I_{\tau+1},\tht_{\tau+1}) \coloneqq \fupd (I_\tau,\tht_\tau,\del_t)$ \label{state:update}
            \State Let $(l_{\tau+1},c_{\tau+1},r_{\tau+1}) \coloneqq \tht_{\tau+1}(I_{\tau+1})$ \label{s:query-points}
            \State Recommend 
            $
                R_t
            \in
                 \argmin_{x\in  \{l_{\tau+1}, c_{\tau+1}, r_{\tau+1}\}} J_{x,t}^+ $ \label{state:recommendationone}
            \State \textbf{break}
        \ElsIf{%
        $B_{\tau,t} \ge \sum_{\tau'=0}^{\tau-1} B_{\tau'}$
        \label{state:recommendation-one-bis}}
            \State Recommend
            $
                R_t
            \in
                \argmin_{x\in \{l_\tau, c_\tau, r_\tau\}} J_{x,t}^+$ \label{state:recommendationonebis}
        \Else \label{state:recommendation-two}
            \State Recommend 
            $
                R_t
            \coloneqq
                R_{t_{\tau-1}}
            $
            \label{state:recommendationtwo}
        \EndIf
    \EndFor
    \label{s:repeat-end}
\EndFor 
\end{algorithmic}
\end{algorithm}

We note that the assignment $\tau_t\coloneqq \tau$ in brackets in \Cref{state:extra-command} is not needed to run the algorithm. 
We only added it for notational convenience of the analysis.

As noted above, by definition of the $\fupd$ function, \algName{} only queries points in the rescaled dyadic mesh $\bcb{ I^- + k \cdot 2^{-h} \cdot \labs I : h \in \N, k \in [2^h-1] }$.
Moreover, we stress that \algName{} is any-time (it does not need to know the time horizon $T$ \emph{a priori}), any-budget (it does not need to know the total budget $B \coloneqq \sum_{t=1}^T b_t$) and does not require the unknown objective to be \lip{}.
Nevertheless, we will show in \Cref{t:upper,t:lower} that its performance is guaranteed to be near-optimal even when compared to algorithms with full knowledge of $T$ and $B$, and run on convex \lip{} functions with \emph{known} \lip{} constant.

\section{Upper bound}
\label{s:upper}

In this section, we provide optimization error guarantees for our \algName{} algorithm.

\begin{theorem}
\label{t:upper}
For any bounded interval $I\s\bbR$, if the optimizer is running \algName{} (\Cref{a:dyadic}) with input $I$ in an environment satisfying Assumption \ref{ass:budget} for some $c\ge 0$ and $\alpha>0$, then, there exist $c_1 \le  12 \cdot \big(\fracc{48}{(2^{1/\alpha}-1)}) ^\alpha , c_2 \le 9/8, c_3 \ge \fracc{(\ln 2)}{48} $ such that, for any time $T\in \N$, every sequence of budgets $b_1,\dots,b_T>0$, and every convex function $f\colon I \to \bbR$, the optimization error $f(R_T) - \inf_{x\in I}f(x)$ is upper bounded by
\begin{equation}
\label{e:upper-bound}
    c_1 \cdot \frac{c}{(\sum_{t=1}^T b_t)^{\alpha}}
    +
    c_2 \cdot  L \labs{I} \exp \lrb{-c_3 \cdot  \frac{\sum_{t=1}^T b_t}{\max_{t \in [T]} b_t}} \;,
\end{equation}
where $L$ is the \lip{} constant of $f$ on $[l_{\tau_T}, r_{\tau_T}]$.
\end{theorem}

The full proof of this result can be found in \Cref{s:appe-proof-ub}.
Before presenting a sketch of it here, we make a few remarks.
First, note that the bound is non-trivial even when the function is \emph{not} globally Lipschitz (as is it the case, e.g., for the function $f(x) = -\sqrt{1-x^2}$ defined on the interval $I = [-1,1]$), since it depends on a \emph{local} Lipschitz constant $L$ (which is always finite) that, informally, as the epochs go by, captures better and better how much the function varies around the points that are close to the minimum.
Second, note that (up to the constants $c_1,c_2,c_3$) the bound consists of two terms.

The first term  $\fracc{c}{(\sum_{t=1}^T b_t)^{\alpha}}$ is a consequence of the fuzziness of the evaluations, that is regulated by \Cref{ass:budget}: when $\sum_{t=1}^T b_t \ge 1$,
it decreases when $\alpha$ increases or $c$ decreases.
Moreover, when $c = 0$ and $b_t=1$ for all $t\in[T]$, our problem reduces to deterministic convex optimization (DCO).
In this case, the first term vanishes completely, leaving behind only the known optimal exponentially-decaying rate $L \labs{I} e^{-\Omega(T)}$ for DCO.

The second term $L \labs{I} e^{- \Omega \lrb{ \fracc{\sum_{t=1}^T b_t}{\max_{t \in [T]} b_t} } }$ is a consequence of the discrete nature of our setting.
Notably, if the optimizer could choose to invest infinitesimally small budgets $b_t$ (i.e., if the discrete optimization protocol became a continuous one), the term would vanish completely.
Strikingly, when this is the case, the bound becomes completely \emph{independent} of the \lip{} constant $L$.
To the best of our knowledge, this is the first result in convex optimization that shows that the dependence on $L$ could be entirely lifted if we could transition from a discrete to a continuous setting.
In other words, our bound gives a parameterization of the dependence on the \lip{} constant in terms of how close our setting is to a continuous one.
The high-level reason for this behavior is that, in a discrete setting, the optimizer might be forced to spend a large amount of budget $b_t$ on a point $X_t$ where a significantly smaller investment would have been sufficient to determine whether or not that point was suboptimal.
In this case, if the function is varying significantly, the number of queries could not be sufficient to get closer to a minimizer, and this would yield an optimization error that scales with $L$.
Finally, we note that, naturally, the \lip{} constant $L$ and the domain length $\labs I$ appear as a product.
Indeed, shrinking (resp., dilating) the domain of a function $f\colon I \to \bbR$ corresponds (inversely-proportionally) to an increase (resp., decrease) of the \lip{} constant and vice versa.

\begin{proof}[Proof sketch.]
The proof of \Cref{t:upper} proceeds as follows.
We divide the analysis in the three cases sketched below, depending on how \algName{} selects $R_T$ in \Cref{state:recommendationone,state:recommendationonebis,state:recommendationtwo}.
\begin{enumerate}
    \item \label{i:first-case} $\del_T \neq \cancNone$. 
    In this case, we partition the number of epochs in several classes and focus our attention on the class where we invested the highest fraction of the total budget $\sum_{t=1}^T b_t$.
    Say that this class contains $n$ epochs.
    If $n$ is small, we show that in the last epoch of this class there exist two query points that are near-optimal and that the recommendation $R_T$ of \algName{} has guarantees that are close to those of these two near-optimal points.
    If, on the other hand, $n$ is large, the result follows by the local \lip{}ness of the objective.
    \item $\del_T = \cancNone$ and the majority of the budget was invested in the last epoch.
    In this case, we split again the analysis in two further cases.
    If the maximum budget $\max_{t\in[T]} b_t$ is small, we show that all three query points of the last epoch are near-optimal, therefore so is the recommendation $R_T$.
    If, on the other hand, the maximum budget $\max_{t\in[T]} b_t$ is large, we fall back again to the local \lip{}ness of the objective.
    \item $\del_T = \cancNone$ and the majority of the budget was invested before the last epoch.
    Since in this case the recommendation $R_T$ is the same as the recommendation that ended the previous epoch, the result follows by \Cref{i:first-case}, using half of the total budget. \qedhere
\end{enumerate}
\end{proof}
The full proof of \Cref{t:upper} is deferred to \Cref{s:appe-proof-ub}.
%

\section{Lower bound}
\label{s:lower}

In this section, we show that \algName{} is worst-case optimal, in the sense that there exist instances where the upper bound of \Cref{t:upper} is matched (up to possibly different constants $c_1,c_2,c_3$).
The apparent asymmetry between our upper and lower bounds is due to the fact that, in \Cref{t:lower}:
\begin{itemize}
    \item We gave the optimizer the freedom to select the time horizon $T$ and total budget $B$ ahead of time.
    \item We restricted the result to the class of convex \lip{} functions.
\end{itemize}
Note that both these changes make our results stronger, since \algName{} is able to match the lower bound despite lacking the freedom to select $T, B$ (in fact, being totally oblivious to a possibly adversarial choice of both) and \Cref{t:upper} holds even for non-\lip{} functions.

\begin{theorem}
\label{t:lower}
For any nondegenerate bounded interval $I\s\bbR$, if the environment satisfies Assumption \ref{ass:budget} for some $c\ge 0$ and $\alpha>0$, then, there exist $c_1 \ge  \nicefrac{1}{4} , c_2 \ge \nicefrac{1}{32e}, c_3 \le 1 $ such that, for any time $T\in \N$, every total budget $B>0$, every Lipschitz constant $L>0$, and every algorithm run by the optimizer, there exist a sequence of budgets $b_1,\dots,b_T$ such that $\sum_{t=1}^T b_t = B$ and there exists a $\max\big(\frac{c}{|I| B^\alpha }, L\big)$-Lipschitz convex function $f\colon I \to \bbR$, for which the optimization error $f(R_T) - \inf_{x\in I}f(x)$ is lower bounded by
\begin{equation}
\label{e:lower-bound}
    c_1 \cdot  \frac{c}{(\sum_{t=1}^T b_t)^{\alpha}}
    +
    c_2 \cdot L \labs{I} \exp \lrb{-c_3 \cdot  \frac{\sum_{t=1}^T b_t}{\max_{t \in [T]} b_t}} \;.
\end{equation}
\end{theorem}

\begin{proof}
    Fix a nondegenerate bounded interval $I\s\bbR$.
    Fix also an horizon $T \in \N$, a total budget $B>0$, and a Lipschitz constant $L > 0$.
    For each $t \in [T]$, define $b_t \coloneqq B/T$.
    We divide the proof in two cases, depending on which of the two addends in \eqref{e:lower-bound} is the dominant term.
    
    Assume first $\frac{1}{4} \cdot  \frac{c}{(\sum_{t=1}^T b_t)^{\alpha}} \ge \frac{1}{32e} \cdot L \labs{I} \exp \lrb{- \frac{\sum_{t=1}^T b_t}{\max_{t \in [T]} b_t}}$.
    For all $b > 0$, define $J(b) \coloneqq \bsb{-\fracc{c}{(2b^\alpha)},\fracc{c}{(2b^\alpha)} }$.
    Consider the two alternative objective functions $f_+$ and $f_-$, defined for all $x \in I$, by 
    \[
        f_{\pm}(x) 
    \coloneqq 
        \pm\lrb{1 - \frac{2(x-I^-)}{\labs{I}} } \cdot \frac{c}{2 B^\alpha}
        \;.
    \]
    At each time $t \in [T]$, if the algorithm chosen by the optimizer queried $X_1,\dots, X_t$, the environment returns the fuzzy evaluation $J_t \coloneqq J(\fB_{X_t,t})$, where we recall that $\fB_{x,t}$ was defined, for any $x \in I$, by $\sum_{s=1}^t b_s \I \{X_s = x\}$.
    Note that the environment satisfies \Cref{ass:budget} and that both functions $f_{\pm}$ are $\frac{c}{|I| B^\alpha }$-Lipschitz.
    Moreover, the algorithm provides the same queries and recommendations for both $f^-$ and $f^+$, as it receives the same $J_1,\dots, J_T$.
    Furthermore, if the algorithm recommends $R_T \ge \fracc{(I^- + I^+)}{2}$ then $f_-(R_T) - \inf_{x\in I}f_-(x) \ge \fracc{c}{(2B^\alpha)}$, while if the algorithm recommends $R_T < \fracc{(I^- + I^+)}{2}$ then $f_+(R_T) - \inf_{x\in I}f_+(x) \ge \fracc{c}{(2B^\alpha)}$. 
    Thus, in both cases there exists a $\frac{c}{|I| B^\alpha }$-Lipschitz convex function $f \in \{ f_-,f_+ \}$ for which:
    \[
        f(R_T) - \inf_{x\in I}f(x) 
    \ge
        \frac{1}{4} \cdot  \frac{c}{(\sum_{t=1}^T b_t)^{\alpha}} + \frac{1}{32e} L \labs I e^{- \frac{\sum_{t=1}^T b_t}{ \max_{t \in [T]} b_t } } \;.
    \]
    
    Assume now $\frac{1}{4} \cdot  \frac{c}{(\sum_{t=1}^T b_t)^{\alpha}} < \frac{1}{32e} \cdot L \labs{I} \exp \lrb{- \frac{\sum_{t=1}^T b_t}{\max_{t \in [T]} b_t}}$.
    In this case, at each time $t \in [T]$, the environment returns $J_t \coloneqq\{ f(X_t) \}$. 
    Note that, in this instance, our problem reduces to \emph{deterministic} convex optimization. 
    By a classic lower bound for deterministic convex optimization (see, e.g., \citealt[Theorem~3.2.8]{nesterov2018}), then, there exists an $L$-Lipschitz convex function $f \colon I \to \bbR$ for which
    \begin{align*}
    &
        f(R_T) - \inf_{x\in I}f(x) 
    > 
        \frac{1}{16e} L \labs I e^{-T} 
    = 
        \frac{1}{16e} L \labs I e^{- \frac{\sum_{t=1}^T b_t}{ \max_{t \in [T]} b_t } }
    \\
    &
    \qquad\qquad
    \ge
        \frac{1}{4} \cdot  \frac{c}{(\sum_{t=1}^T b_t)^{\alpha}} + \frac{1}{32e} L \labs I e^{- \frac{\sum_{t=1}^T b_t}{ \max_{t \in [T]} b_t } }
        \;.
    \end{align*}
    Being the interval $I$, the horizon $T$, the budget $B$, and the Lipschitz constant $L$ arbitrarily chosen, the theorem follows.
\end{proof}

\section{Experiments}\label{s:experiments}

In this section we illustrate our theoretical findings with an experimental evaluation. In particular, the purpose of the numerical experiments is to show that:

\begin{enumerate}
\item The dependence on the local Lipschitz constants in the upper bound of \Cref{t:upper} provides a concrete advantage over a much simpler dependence on the global Lipschitz constant. 

\item The lower bound of \Cref{t:lower} is quite tight in practice. 
\end{enumerate}

Each experiment is implemented by specifying the oracle in terms of the function $f$, the parameters $c, \alpha, b_t$ and the location of $f(x)$ within the intervals $J_{x, t}$.

\subsection{Adaptive Lipschitz constants}
This experiment aims to show the advantages of featuring the local Lipschitz constant, as given in Theorem 1, over a bound that uses the global constant. To this end, we optimize the function $f(x) = -\sqrt{x} + 1$ over the interval $(a, 1)$ with $a > 0$. The Lipschitz constant of $f$ grows roughly as $\nicefrac{1}{\sqrt{a}}$. Moreover, $f(x^*) = 0$ and the maximum value is smaller than $1$, thus any meaningful upper bound on the optimization error should be smaller that $1$. In the experiment we set $a=0.001$, $c=0.1, \alpha=1, b_t=1$ (for all $t$), and we let the intervals $J_{x,t}$ to be symmetric around $f(x)$. We run the Dyadic Search for $T=1000$ iterations. 

\cref{fig:ub} shows that the upper bound featuring the local Lipschitz constants (local) is much tighter than that featuring the global constant (global). In particular, as denoted by the vertical lines, the former falls below 1 after 84 iterations, while the latter becomes non-trivial only at the iteration $339$. 

\subsection{Tightness of the lower bound}
This experiment aims to show the tightness of our lower bound. The oracle is based on the setting described in the first part of the proof of Theorem 2 with parameters $c=0.1, \alpha=1/2, b_t = 1$ for all $t$. 
Since the lower bound holds in the finite horizon setting, we set $T=1000$, and for each $t \in [T]$ we run the Dyadic Search for $t$ iterations; then we measure the optimization error $f(R_t)-f(x^*)$. Notice that for each $t$, the algorithm actually solves a different instance of the problem where the overall budget $B$ is set to $t$.

\Cref{fig:lb} shows the optimization errors of Dyadic Search, as measured at the end of each run of $t$ iterations, along with the lower and the upper bounds. 
The actual performance of the algorithm is relatively close to the lower bound and essentially follows the same pattern.

\begin{figure*}[t]
\centering
\subfigure[Local vs. Global bounds]{\label{fig:ub}\includegraphics[width=0.49\textwidth]{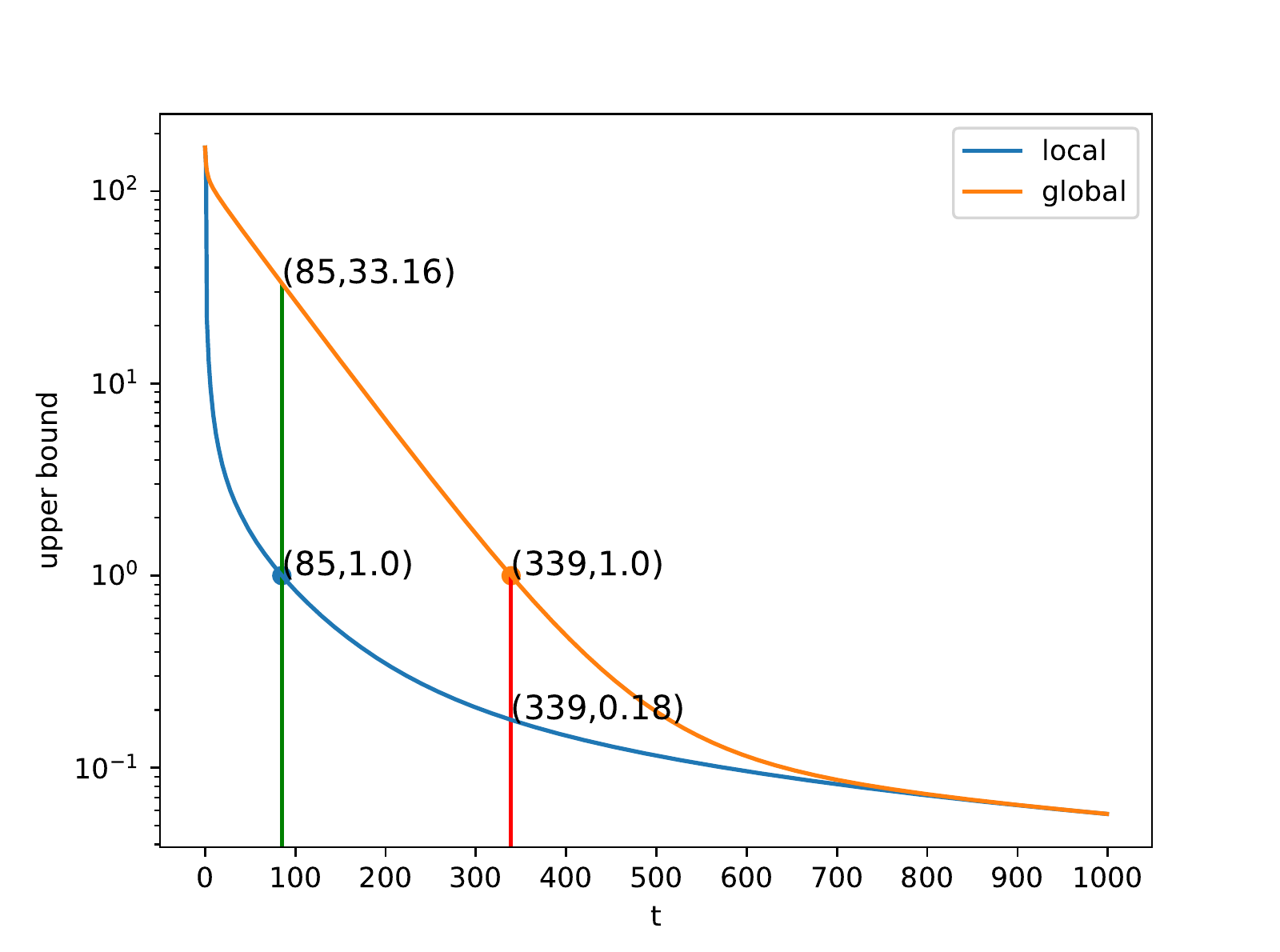}}
\subfigure[Confidence error bounds]{\label{fig:lb}\includegraphics[width=0.49\textwidth]{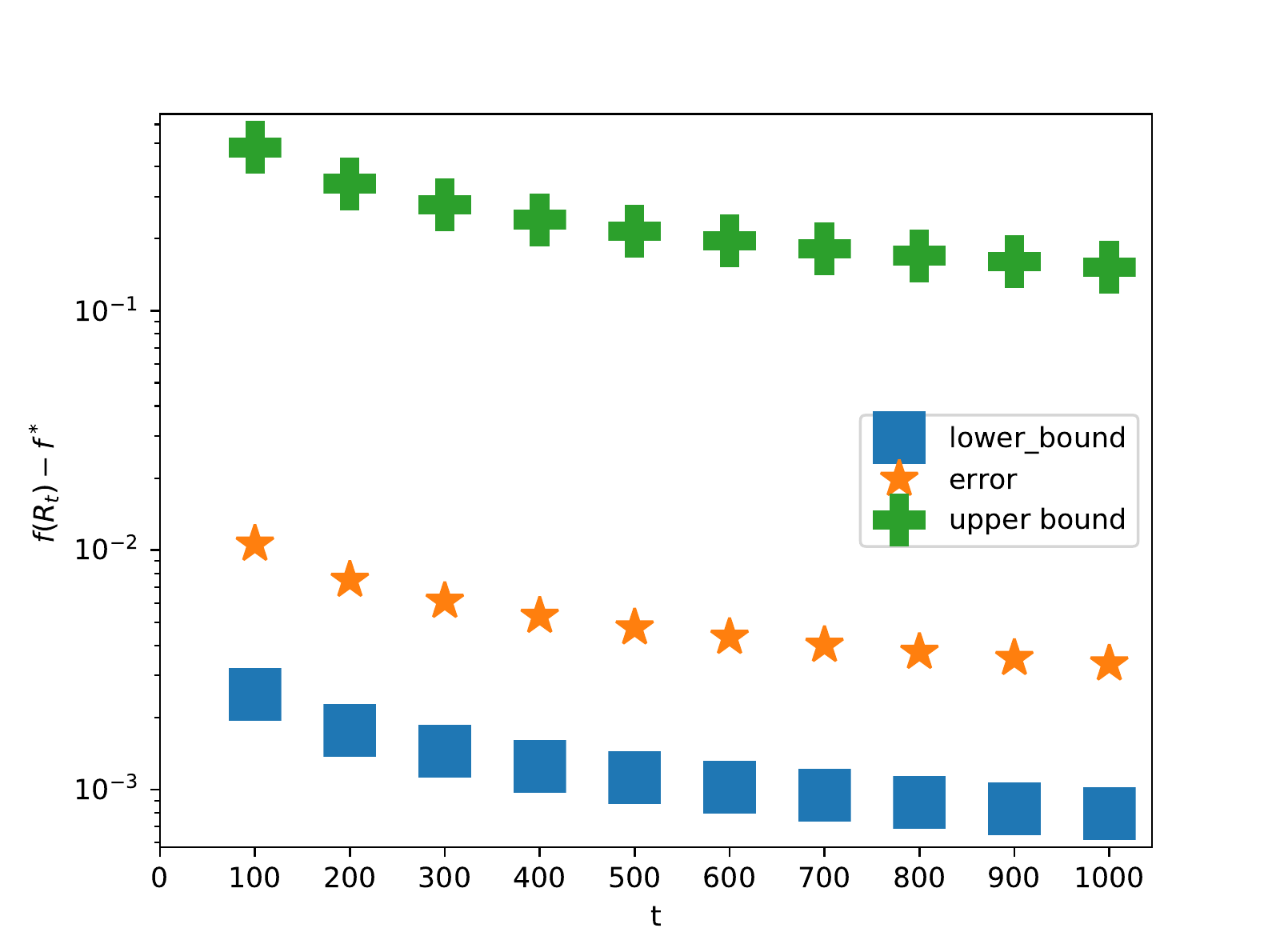}}
\caption{Experimental results. In (a) the vertical green line raises at the first point where the local falls below $1$; the red line raises at the first point where global falls below one. In (b) each triple of point (blue square, orange star, green cross) denotes respectively the lower bound, the actual error of Dyadic Search and the upper bound on the given instance.}
\label{fig:exps}
\end{figure*}

\section{Conclusions}\label{s:conclusions}

We studied a budget version of the classic zeroth-order univariate online convex optimization problem. 
We designed an any-time, any-budget algorithm, called Dyadic Search (\Cref{a:dyadic}), that does not require the \emph{a priori} knowledge of the (local) \lip{} constant of the objective and works even in an adversarial and adaptive environment. 
We proved theoretical optimization error guarantees for Dyadic Search (\Cref{t:upper}) that scale with the quality of the budget (\Cref{ass:budget}) and the relevant parameters of the problem (domain size and local \lip{}ness). 
Additionally, we showed that no algorithm can significantly outperform Dyadic Search, not even if it had the freedom to select the time horizon $T$ and the total budget $B$ (\Cref{t:lower}). 
As a side-effect of our analysis, we obtain that the closer to a continuous setting our optimization problem is (i.e., the smaller the allowed budgets $b_t$ are), the milder the dependence on the local \lip{}ness of the objective, vanishing at the limit.
To our knowledge, this is the first analysis of a convex optimization problem that shows (in the limit of the infinitesimally fine budget granularity) that no dependence on any \lip{} constant is needed in upper bounds.

\appendix

\bibliographystyle{plainnat}
\bibliography{biblio}

\section*{Acknowledgments}

Tommaso Cesari gratefully acknowledges the support of IBM.

\section{Full proof of Theorem~\ref{t:upper}}
\label{s:appe-proof-ub}

In this section, we give a detailed proof of Theorem~\ref{t:upper}.

\begin{proof}[Proof of \Cref{t:upper}]
Fix any bounded interval $I\s \bbR$, a time $T\in \N$, and a convex function $f\colon I \to \bbR$.
Without loss of generality, we can (and do!) assume that $I$ contains at least two distinct points.\footnote{Otherwise, the optimization error is trivially zero.}
Moreover, without loss of generality, we can also (and do!) assume that $f$ attains its minimum in $I$.\footnote{%
Indeed, if it does not, then there exists $x^\star \in \{I^-,I^+\}$ such that $\lim_{x\to x^\star, x \in I} f(x) = \inf_{x\in I}f(x)$ (this can only happen if $I$ is not closed or $f$ is discontinuous at $x^\star$; in the latter case, note that by convexity, $f(x^\star) > \lim_{x\to x^\star, x \in I} f(x) = \inf_{x\in I}f(x)$).
Thus, noting that \algName{} \emph{never} queries nor recommends the endpoints $\{I^-,I^+\}$ of $I$, one can replace $f$ with $\bar f$, where $\bar f(x) \coloneqq f(x)$ for all $x \in I$ and $\bar f (x^\star) \coloneqq \inf_{x\in I}f(x)$.
This way, up to extending (or redefining) $f$ at $x^\star$, we are left with a convex function $\bar f$ such that $\bar{f}(X_1) \in J_1,\dots, \bar{f}(X_T) \in J_T$, attains its minimum at $x^\star$ (in its domain), and satisfies $f(R_T) - \inf_{x\in I}f(x) = \bar f(R_T) - \bar f(x^\star)$.%
}
Then, note that the active interval $I_\tau$ of any epoch $\tau \in [\tau_T]$ (defined in the initialization and updated at \Cref{state:update}) always contains at least a minimizer, because the first active interval $I_1$ is the entire domain $I$ and, by the unimodality of $f$ and the definition of the $\fdel$ and $\fupd$ functions, \algName{} deletes a fraction of the active interval (\Cref{state:update}) only if it is certain that the value of $f$ at one of the remaining points is no bigger than \emph{all} of the deleted points.
Thus, there exists (and we fix for the rest of the proof) an $x^\star \in I$ such that  $x^\star \in I_{\tau_T} \s \dots \s I_1 $ and $f(x^\star) = \min (f)$.
Recall that, for all $t\in \N$, $\tau_t$ is the epoch of round $t$ (\Cref{state:extra-command} of \Cref{a:dyadic}).
Also, for the sake of convenience, we define $t_0 \coloneqq 0$ and, if the last epoch is not concluded exactly at the end of time $T$, we redefine $t_{\tau_T} \coloneqq T$ and $B_{\tau_T} \coloneqq B_{\tau_T,T}$.

To prove the result, we analyze separately the performance of the recommendation $R_T$ of \algName{} in the three cases of \Cref{state:recommendationone,state:recommendationonebis,state:recommendationtwo}.

Assume at first that $\del_T \neq \cancNone$ (i.e., the condition on \Cref{state:recommendation-one} is true and we recommend $R_T$ as in \Cref{state:recommendationone}).
We partition the epochs $\tau \in [\tau_T]$ into four sets, depending on whether or not the epoch is uniform and whether or not $x^\star\le c_\tau$. 
More precisely, for any $\zeta \in \{\unif,\nunif\}$ and $\vdash \; \in \{\le, >\}$, we let $A_{\zeta,\vdash}$ be the set of all epochs $\tau \in [\tau_T]$ such that $\tht_\tau = \zeta$ and 
$x^\star \vdash c_\tau$.
(or, equivalently stated, that there exist at least two distinct elements $x,y\in \{l_\tau, c_\tau, r_\tau\}$ such that $x^\star \vdash x \vdash y$).
Now fix 
$
A
\coloneqq
    A_{\zeta,\vdash}
$,
where
\[
    (\zeta,\vdash)
\in 
    \argmax_{(\zeta',\vdash')\in \{\unif,\nunif\} \times \{\le, >\}} \sum_{\tau \in A_{\zeta',\vdash'}} B_\tau \;.
\]
In words, $A$ is the set of epochs $A_{\zeta',\vdash'}$ where the Dyadic Search spent the highest budget.
Define, for each $\tau \in A$, the points $x_\tau \neq y_\tau$ as the closest and second-closest points in $\{l_\tau, c_\tau, r_\tau\}$ to $x^\star$ such that $x^\star \vdash x_\tau \vdash y_\tau$ (they always exist by definition of $A$).
More precisely,
\begin{align*}
    x_\tau & \coloneqq \argmin_{x \in \{l_\tau, c_\tau, r_\tau\}, \, x^\star \vdash x}  \labs{x - x^\star } \;, 
\\
    y_\tau & \coloneqq \argmin_{x \in \{l_\tau, c_\tau, r_\tau\}\m\{x_\tau\}, \, x^\star \vdash x}  \labs{x - x^\star } 
    \;.
\end{align*}
Let $n$ be the number of elements of $A$ and $\kappa_1, \dots, \kappa_n$ be the elements of $A$ in increasing order.
Then, for any $i\in\N$, we have
\begin{align}
    \labs{ I_{\kappa_{i}} } & \le \labs{ I_{\kappa_{i-1}} }/2 && \text{ if } 2 \le i \le n \label{align:five}\;, \\
    \labs{ y_{\kappa_i} - x_{\kappa_i} } & \le \labs{ y_{\kappa_{i-1}} - x_{\kappa_{i-1}} }/2 && \text{ if } 2 \le i \le n \;, \label{align:one} \\
    \labs{ x_{\kappa_i} - x^\star } & \le \labs{ y_{\kappa_i} - x_{\kappa_i} } \cdot 2 && \text{ if } i \le n \;, \label{align:two} \\
    \labs{ x_{\kappa_{i}} - x^\star } & \le \labs{ x_{\kappa_{i-1}} - x^\star } && \text{ if } 2 \le i \le n \label{align:three} \;, \\
    \labs{ y_{\kappa_{i}} - x^\star } & \le \labs{ y_{\kappa_{i-1}} - x^\star } && \text{ if } 2 \le i \le n \label{align:four}\;.
\end{align}
Here, \eqref{align:five} follows directly by the definition of the $\fupd$ function, noting that there are never two uniform or two non-uniform epochs in a row, unless half of the current interval is eliminated in one single call of the $\fupd$ function. 
Inequality \eqref{align:one} follows directly from \eqref{align:five}. Inequality \eqref{align:two} is a consequence of the definitions of the partition functions $\unif$ and $\nunif$.
To prove \eqref{align:three}, note first that the claim holds trivially when $x_{\kappa_{i-1}} \in \{l_{\kappa_{i}},c_{\kappa_{i}},r_{\kappa_{i}}\}$.
When this is not the case, since \algName{} discarded $x_{\kappa_{i-1}}$ at the end of some previous epoch, it either holds that $I_{\kappa_{i}} \subset (-\iop,x_{\kappa_{i-1}}]$ or $I_{\kappa_{i}} \subset [x_{\kappa_{i-1}},\iop)$. If $x^\star \le x_{\kappa_{i-1}}$ (meaning that $\vdash = \le$), it follows from $x^\star \in I_{\kappa_{i}}$ that $I_{\kappa_{i}} \subset(-\iop,x_{\kappa_{i-1}}]$, which implies $x^\star \le x_{\kappa_{i}} < x_{\kappa_{i-1}}$.
Analogously, if $x^\star > x_{\kappa_{i-1}}$ (meaning that $\vdash = >$), it follows from $x^\star \in I_{\kappa_{i}}$ that $I_{\kappa_{i}} \subset [x_{\kappa_{i-1}},\iop)$, which implies $x^\star > x_{\kappa_{i}} > x_{\kappa_{i-1}}$.
This proves \eqref{align:three}.
Moreover, as a direct consequence of \eqref{align:one} and \eqref{align:three}, we obtain \eqref{align:four}. 

By construction, we have that
\begin{equation}
\label{eq:greater_budget}
    4 \sum_{\tau \in A} B_\tau
\ge
    \sum_{\tau \in [\tau_T] } B_\tau
=
    \sum_{t=1}^T b_t \;.    
\end{equation}
Now, we show that for any $\tau \in [\tau_T]$ and $k \in \{0,\dots,t_\tau\}$, we have
\begin{equation}
\label{eq:equipartition}
    \min_{x\in \{l_\tau, c_\tau, r_\tau\}} \fB_{x,\, t_\tau -k}
\ge
    \frac {B_\tau - (2+k) \max_{t \in [T]} b_t} 3 \;,
\end{equation}
i.e., that the total budget $\fB_{x,\, t_\tau -k}$ spent on any query point $x\in \{l_\tau, c_\tau, r_\tau\}$ up to time $t_\tau - k$ is no smaller (up to $\Theta (k) \cdot \max_{t \in [T]} b_t $) than the budget $B_\tau$ spent (on all query points) only during epoch $\tau$.
Indeed, for any $\tau \in [\tau_T]$ and $k \in \{0,\dots,t_\tau\}$, letting $x_{\min} \in \argmin_{{x\in \{l_\tau, c_\tau, r_\tau\}}} \fB_{x,\, t_\tau -k}$ be a query point where the algorithm spent the least amount of budget up to time $t_\tau - k$ and $M \coloneqq \bcb{ x\in \{l_\tau, c_\tau, r_\tau\} : \fB_{x,t_\tau -k} - \fB_{x_{\min},t_\tau -k} \le  \max_{t \in [T]} b_t }$ be the set of all query points in which the algorithm spent a budget close to that of $x_{\min}$, we have
\begin{align*}
&
    3 \cdot \min_{x\in \{l_\tau, c_\tau, r_\tau\}} \fB_{x,\, t_\tau -k}
\overset{\phantom{(*)}}{\ge}
    \sum_{ x\in M } 
    \fB_{x, t_\tau - k} -2  \max_{t \in [T]} b_t
\\
&
\qquad\qquad
\overset{\phantom{(*)}}{\ge}
    \sum_{ x\in M } 
    \sum_{t = t_{\tau-1} +1 }^{t_\tau - k} b_t \I \{X_t = x \} 
    - 2  \max_{t \in [T]} b_t
\\
&
\qquad\qquad
\overset{(*)}{=}
    B_\tau - \sum_{t = t_\tau - k + 1}^{t_\tau} b_t - 2  \max_{t \in [T]} b_t
\\
&
\qquad\qquad
\overset{\phantom{(*)}}{\ge}
    B_\tau - (2+k)  \max_{t \in [T]} b_t
\end{align*}
with the understanding that any sum $\sum_{s=i}^j z_s$ is equal to zero whenever $i > j$, and where in $(*)$ we used the fact that, if $x \in \{l_\tau, c_\tau, r_\tau\}$ is such that $\fB_{x,t_\tau -k} - \fB_{x_{\min},t_\tau -k} > \max_{t \in [T]} b_t$, then \algName{} never queried $x$ in epoch $\tau$ up to time $t_\tau - k$, i.e., $\sum_{t = t_{\tau-1} + 1}^{t_\tau - k} b_t = 0$.
This proves \eqref{eq:equipartition}.

Thus, for any $\tau \in A$, if $B_\tau > 3  \max_{t \in [T]} b_t$, since $J^+_{x_\tau, t_\tau - 1} > J^-_{y_\tau, t_\tau - 1}$ (this follows from the definition of the $\fdel$ function and can be proved by exhaustion), it holds that
\begin{align}
&
    f(y_\tau)-f(x_\tau)
\le
    J^+_{y_\tau, t_\tau - 1} - J^-_{x_\tau, t_\tau - 1}
    \nonumber
\\
&
\quad\overset{\phantom{\eqref{eq:equipartition}}}{<}
    J^+_{y_\tau, t_\tau - 1} - J^-_{y_\tau, t_\tau - 1} + J^+_{x_\tau, t_\tau - 1} - J^-_{x_\tau, t_\tau - 1}
    \nonumber
\\
&
\quad\overset{\phantom{\eqref{eq:equipartition}}}{=}
    \labs{J_{y_\tau, t_\tau - 1}} + \labs{J_{x_\tau, t_\tau - 1}}
\le
    \frac{c}{\fB_{y_\tau,\, t_\tau -1}^\alpha} + \frac{c}{\fB_{x_\tau,\, t_\tau -1}^\alpha}
    \nonumber
\\
&
\quad\overset{\phantom{\eqref{eq:equipartition}}}{\le}
    \frac{2 \cdot c}{\lrb{ \min_{x\in \{l_\tau, c_\tau, r_\tau\}} \fB_{x,\, t_\tau -1} }^\alpha}
    \nonumber
\\
&
\quad\overset{\eqref{eq:equipartition}}{\le}
    \frac{3^\alpha \cdot 2 \cdot c}{\brb{B_\tau - 3 \max_{t \in [T]} b_t }^\alpha} \;.
    \label{e:giacomo}
\end{align}
Assume now that $f(y_{\kappa_n}) - f(x_{\kappa_n}) > 0$.
Then, for any $\tau \in A$, by convexity and inequalities \eqref{align:three}--\eqref{align:four}, we have that $f(y_\tau) - f(x_\tau) > 0$ is also true.
By \eqref{e:giacomo}, it follows that, for any $i \in [n]$, regardless of the fact that the inequality $B_{\kappa_i} > 3 \max_{t \in [T]} b_t$ holds or not,
\begin{equation}
\label{eq:budget_bound_1}
    B_{\kappa_i}
\le
    3 \max_{t \in [T]} b_t + \frac{3 \cdot (2 \cdot c)^{1/\alpha}}{ \brb{ f(y_{\kappa_i}) - f(x_{\kappa_i}) }^{1/\alpha} } \;.
\end{equation}
Summing \eqref{eq:budget_bound_1} over $i \in [a]$, we obtain
\begin{align}
    \sum_{t=1}^T b_t
\overset{\eqref{eq:greater_budget}}&{\le}
    4 \sum_{\tau \in A} B_\tau
=
    4 \sum_{i = 1}^n B_{\kappa_i}
    \nonumber
\\
\overset{\eqref{eq:budget_bound_1}}&{\le}
    4 \sum_{i = 1} ^ n \lrb{ 3 \max_{t \in [T]} b_t + \frac{3 \cdot (2 \cdot c)^{1/\alpha}}{ \brb{ f(y_{\kappa_i}) - f(x_{\kappa_i}) }^{1/\alpha} } }
\nonumber
\\
\overset{}&{=}
    12 \cdot \max_{t \in [T]} b_t \cdot n + \sum_{i = 1 }^n \frac{12 \cdot (2 \cdot c)^{1/\alpha}}{ \brb{ f(y_{\kappa_i}) - f(x_{\kappa_i}) }^{1/\alpha} } \;.
    \label{e:random-inequality}
\end{align}
Now, using \eqref{align:three}--\eqref{align:four} together with the fact that difference quotients of a convex function are non-decreasing in both variables, and since for each $i \in [n]$ it holds that $(y_{\kappa_n}-x_{\kappa_n})\cdot (y_{\kappa_i}-x_{\kappa_i})>0$, we have
\begin{align}
&
    \sum_{i = 1 }^n \frac{ \brb{ f(y_{\kappa_n}) - f(x_{\kappa_n}) }^{1/\alpha} }{ \brb{ f(y_{\kappa_i}) - f(x_{\kappa_i}) }^{1/\alpha} }
=
    \sum_{i = 1 }^n \lrb{ \frac{f(y_{\kappa_n}) - f(x_{\kappa_n})}{ f(y_{\kappa_i}) - f(x_{\kappa_i}) } }^{1/\alpha}
    \nonumber
\\
\overset{}&{=}
    \sum_{i = 1 }^n  \lrb{ \lrb{ \frac{ \frac{f(y_{\kappa_n}) - f(x_{\kappa_n})}{ y_{\kappa_n} - x_{\kappa_n}} }{ \frac{ f(y_{\kappa_i}) - f(x_{\kappa_i}) }{ y_{\kappa_i} - x_{\kappa_i} } }  }^{1/\alpha} \cdot \labs{ \frac{ y_{\kappa_n} - x_{\kappa_n} }{ y_{\kappa_i} - x_{\kappa_i} } } ^{1/\alpha} }
    \nonumber
\\
\overset{\eqref{align:three}+\eqref{align:four}}&{\le}
    \sum_{i = 1 }^n   \labs{ \frac{ y_{\kappa_n} - x_{\kappa_n} }{ y_{\kappa_i} - x_{\kappa_i} } } ^{1/\alpha} 
\overset{\eqref{align:one}}{\le}
    \sum_{i = 1 }^n   \lrb{ \frac{1}{2^{n-i}}} ^{1/\alpha}
\label{eq:chaining}
\overset{\phantom{\eqref{align:one}}}{\le}
    \frac{1}{1-2^{-1/\alpha}} 
    \;.
\end{align}
Putting \eqref{e:random-inequality} and \eqref{eq:chaining} together, we obtain the inequality
\[
    \sum_{t=1}^T b_t
\le
    12 \cdot \max_{t \in [T]} b_t \cdot n + \frac{ 12 \cdot (2 \cdot c)^{\fracc{1}{\alpha}} \cdot \frac{1}{1-2^{-\fracc{1}{\alpha}}} }{\brb{ f(y_{\kappa_n}) - f(x_{\kappa_n}) }^{1/\alpha}} \;,
\]
that can be rearranged, whenever $\sum_{t=1}^T b_t \ge 24 \cdot \max_{t \in [T]} b_t \cdot n$, to obtain the inequality
\begin{equation}
\label{eq:y_vs_x}
    f(y_{\kappa_n}) - f(x_{\kappa_n})
\le
    4  \cdot\lrb{ \frac{24}{2^{1/\alpha} - 1} } ^\alpha \cdot \frac{c}{ (\sum_{t=1}^T b_t)^\alpha } \;.
\end{equation}
Then, relying again on the fact that difference quotients of a convex function are non-decreasing in both variables, and that $(y_{k_n}-x_{k_n})\cdot (x_{k_n}-x^\star)\ge0$, whenever $\sum_{t=1}^T b_t \ge 24 \cdot \max_{t \in [T]} b_t \cdot n$ and $x_{\kappa_n}\neq x^\star$, we have that

\begin{align}
&
    f(x_{\kappa_n}) - f(x^\star)
\overset{\phantom{\eqref{eq:y_vs_x}}}{=}
    \frac{f(x_{\kappa_n}) - f(x^\star)}{x_{\kappa_n}-x^\star} \cdot (x_{\kappa_n}-x^\star)
    \nonumber
\\
\overset{}&{\le}
    \frac{f(y_{\kappa_n}) - f(x_{\kappa_n})}{y_{\kappa_n}-x_{\kappa_n}} \cdot (x_{\kappa_n}-x^\star)
\nonumber
\\
\overset{\eqref{eq:y_vs_x}}&{\le}
    4 \cdot\lrb{ \frac{24}{2^{1/\alpha} - 1} } ^\alpha \cdot \frac{c}{ (\sum_{t=1}^T b_t)^\alpha } \cdot \labs{\frac{ x_{\kappa_n} - x^\star }{ y_{\kappa_n} - x_{\kappa_n} }}
\nonumber
\\
\overset{\eqref{align:two}}&{\le}
\label{eq:x_vs_xstar}
    8 \cdot \lrb{ \frac{24}{2^{1/\alpha} - 1} } ^\alpha \cdot \frac{c}{ (\sum_{t=1}^T b_t)^\alpha }
\end{align}
(note that \eqref{eq:x_vs_xstar} is trivially true also when $x_{\kappa_n} = x^\star$)
and
\begin{align}
\label{eq:y_vs_xstar}
&
    f(y_{\kappa_n}) - f(x^\star)
=
    f(y_{\kappa_n}) - f(x_{\kappa_n}) + f(x_{\kappa_n}) - f(x^\star)
\nonumber
\\
&\qquad\overset{\eqref{eq:y_vs_x} + \eqref{eq:x_vs_xstar}}{\le}
     12 \cdot \lrb{ \frac{24}{2^{1/\alpha} - 1} } ^\alpha \cdot \frac{c}{ (\sum_{t=1}^T b_t)^\alpha } \;.
\end{align}
Recall that the derivations for \eqref{eq:x_vs_xstar} and \eqref{eq:y_vs_xstar} were carried out under the assumption that $f(y_{\kappa_n}) - f(x_{\kappa_n}) > 0$.
If this assumption does not hold, by convexity, $f(y_{\kappa_n}) = f(x_{\kappa_n}) = f(x^\star)$, therefore \eqref{eq:x_vs_xstar} and \eqref{eq:y_vs_xstar} are still (trivially) true.

Now, we will prove that the recommendation $R_T$ is approximately at least as good as $x_{\kappa_n}$ and $y_{\kappa_n}$.
Since under the assumption that $\sum_{t=1}^T b_t \ge 24 \cdot n \cdot \max_{t \in [T]} b_t$ we showed in \eqref{eq:x_vs_xstar} and \eqref{eq:y_vs_xstar} that $x_{\kappa_n}$ and $y_{\kappa_n}$ are both near-minimizers, this will yield under the same assumption that $R_T$ is also a near-minimizer.

Recalling that we are currently assuming $\del_T \neq \cancNone$, we have that $R_T \in \argmin_{x\in \{l_{\tau_T+1}, c_{\tau_T+1}, r_{\tau_T+1}\}} J_{x,T}^+$, which (as can be checked directly) implies in turn that $R_T \in  \argmin_{x\in \{l_{\tau_T}, c_{\tau_T}, r_{\tau_T}\}} J_{x,T}^+$ and, whenever $\sum_{t=1}^T b_t \ge 24 \cdot n \cdot \max_{t \in [T]} b_t$:
\begin{enumerate}
    \item If $R_T \in \{x_{\kappa_n}, y_{\kappa_n}\}$, then \eqref{e:upper-bound} follows by \eqref{eq:x_vs_xstar} and \eqref{eq:y_vs_xstar}.
    \item If $R_T \notin \{x_{\kappa_n}, y_{\kappa_n}\}$, and $\{x_{\kappa_n}, y_{\kappa_n}\} \s \{l_{\tau_T}, c_{\tau_T}, r_{\tau_T}\}$, then there exists $x\in \{l_{\tau_T}, c_{\tau_T}, r_{\tau_T}\} \m \{R_T\} = \{x_{\kappa_n}, y_{\kappa_n}\}$ such that $f(R_T) \le J_{R_T,T}^+ \le J_{x,T}^- \le f(x)$;
    therefore, \eqref{e:upper-bound} follows by \eqref{eq:x_vs_xstar} and \eqref{eq:y_vs_xstar}.
    \item If $R_T \notin \{x_{\kappa_n}, y_{\kappa_n}\}$, and $\{x_{\kappa_n}, y_{\kappa_n}\} \nsubset \{l_{\tau_T}, c_{\tau_T}, r_{\tau_T}\}$, then, 
    since at least one between $x_{\kappa_n}$ and $y_{\kappa_n}$ does not belong to the set of active query points $\{l_{\tau_T}, c_{\tau_T}, r_{\tau_T}\}$ at time $T$, there exist a past time $t \in [T]$ and a past query point $x \in \{l_{\tau_t}, c_{\tau_t}, r_{\tau_t}\}$ such that $J_{x,t}^+ \le \max \brb{ J_{x_{\kappa_n},t}^-, \, J_{y_{\kappa_n},t}^- }$;
    therefore, noting that the sequence $s \mapsto \min_{x'\in\{l_{\tau_s}, c_{\tau_s}, r_{\tau_s} \}} J_{x',s}^+$ is non-increasing, 
     we have
    $
        f (R_T)
    \le
        J_{R_T,T}^+
    \le
        J_{x,t}^+
    \le
        \max \brb{ J_{x_{\kappa_n, t}}^-,\, J_{y_{\kappa_n,t}}^- }
    \le
        \max \brb{ f(x_{\kappa_n}) , \, f(y_{\kappa_n})}
    $
    and \eqref{e:upper-bound} follows by \eqref{eq:x_vs_xstar} and \eqref{eq:y_vs_xstar}.
\end{enumerate}
On the other hand, if $\sum_{t=1}^T b_t < 24 \cdot  n \cdot  \max_{t \in [T]} b_t$, then
\begin{align*}
&
    f(R_T) - f(x^\star)
\le
    \frac{3}{4} L \labs{I_{\tau_T+1}}
\le
    \frac{9}{16} L \labs{I_{\kappa_n}}
\\
& \qquad
\overset{\eqref{align:five}}{\le}
    \frac{9}{16} L \labs{I} \lrb{\nicefrac{1}{2}}^{n-1}
\le
    \frac{9}{16} L \labs I \lrb{\nicefrac12}^{\frac{ \sum_{t=1}^T b_t  }{24 \max_{t \in [T]} b_t}-1}
\\
& \qquad
\overset{\phantom{\eqref{align:five}}}{=}
    \frac{9}{8}L \labs I \exp \lrb{- \frac{\ln 2}{24} \frac{\sum_{t=1}^T b_t} {\max_{t \in [T]} b_t} } \;,
\end{align*}
where $L$ is the smallest between the \lip{} constants of $f$ on $[l_{\tau_T}, r_{\tau_T}]$ and $[l_{\tau_T+1}, r_{\tau_T+1}]$ ---indeed, by convexity, $L$ is a \lip{} constant for $f$ on the convex hull of $\{x^\star, l_{\tau_T}, r_{\tau_T} \}$ (resp., $\{x^\star, l_{\tau_{T+1}}, r_{\tau_{T+1}} \}$) if and only if it is a \lip{} constant on $[l_{\tau_T}, r_{\tau_T}]$ (resp., $[l_{\tau_{T+1}}, r_{\tau_{T+1}}]$).
Putting everything together yields \eqref{e:upper-bound} when $\del_T \neq \cancNone$.

Assume now that $\del_T = \cancNone$ and $B_{\tau_T} \ge \sum_{\tau' = 0}^{\tau_T -1} B_{\tau'}$ (i.e., the condition on \Cref{state:recommendation-one-bis} is true and we recommend $R_T$ as in \Cref{state:recommendationonebis}).
This implies that the three intervals $J_{l_{\tau_T}, T}, J_{c_{\tau_T}, T}, J_{r_{\tau_T}, T}$ have non-empty intersection, which in turn implies that
\begin{multline}
    \max_{x' \in \{l_{\tau_T}, c_{\tau_T}, r_{\tau_T}\}} J_{x',T}^+  - \min_{x' \in \{l_{\tau_T}, c_{\tau_T}, r_{\tau_T}\}} J_{x',T}^-
\\
\quad \quad
\le
    2 \max_{x' \in \{l_{\tau_T}, c_{\tau_T}, r_{\tau_T}\}} \labs{ J_{x',T} } \;.
\label{eq:skewer}
\end{multline}
Now, we define $f_1,f_2,f_3,f_4$ as the four functions whose graphs are straight lines such that $f_1$ passes through the points $(l_{\tau_T},J^-_{l_{\tau_T},T})$ and $(r_{\tau_T},J^+_{r_{\tau_T},T})$, $f_2$ passes through the points $(c_{\tau_T},J^-_{c_\tau,T})$ and $(r_{\tau_T},J^+_{r_{\tau_T},T})$, $f_3$ passes through the points $(l_{\tau_T},J^+_{l_{\tau_T},T})$ and $(c_{\tau_T},J^-_{c_{\tau_T},T})$, and $f_4$ passes through the points $(l_{\tau_T},J^+_{l_{\tau_T},T})$ and $(r_{\tau_T},J^-_{r_{\tau_T},T})$ (\Cref{f:lines}).
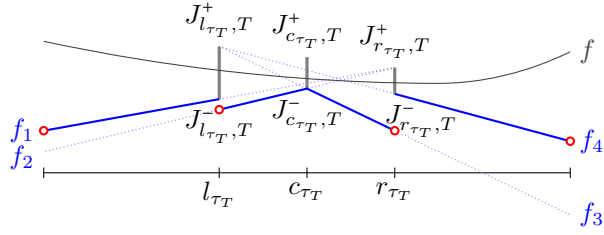
\begin{figure}
    \centering
    \begin{tikzpicture}[scale=7]
    \draw (0,0) -- (1,0);
    \foreach \x in {0, 1/3, 1/2, 2/3, 1}
    { 
        \draw (\x, -0.01) -- (\x, 0.01);
    }
    \draw[very thick, gray] 
        (1/3, 0.14) -- (1/3, 0.24)
        (1/2, 0.16) -- (1/2, 0.22)
        (2/3, 0.15) -- (2/3, 0.20)
    ;
    \draw[densely dotted, blue!50!white] 
        (1/3, 0.14) -- (2/3, 0.20) -- (1/2, 0.16)
        (1/3, 0.12) -- ( 0 , 0.04)
        (2/3, 0.15) -- (1/3, 0.24) -- (1/2, 0.16) 
        (2/3, 0.08) -- ( 1 ,-0.08)
    ;
    \draw[thick, blue] 
        ( 0 , 0.08) -- (1/3, 0.14)
        (1/3, 0.12) -- (1/2, 0.16) -- (2/3, 0.08)
        ( 1 , 0.06) -- (2/3, 0.15)
    ;
    \draw[red, thick, fill=white] 
        (0, 0.08) circle (0.2pt)
        (1/3, 0.12) circle (0.2pt)
        (2/3, 0.08) circle (0.2pt)
        (1, 0.06) circle (0.2pt)
    ;
    \draw 
        (1/3, 0.14) node[below] {$J_{l_{\tau_T},T}^-$}
        (1/3, 0.24) node[above] {$J_{l_{\tau_T},T}^+$}
        (1/2, 0.16) node[below] {$J_{c_{\tau_T},T}^-$}
        (1/2, 0.22) node[above] {$J_{c_{\tau_T},T}^+$}
        ({2/3 + 0.05}, 0.15) node[below] {$J_{r_{\tau_T},T}^-$}
        (2/3, 0.20) node[above] {$J_{r_{\tau_T},T}^+$}
        (1/3, 0) node[below] {$l_{\tau_T}$}
        (1/2, 0) node[below] {$c_{\tau_T}$}
        (2/3, 0) node[below] {$r_{\tau_T}$}
        (0, 0.08) node[left, blue] {$f_1$}
        (0, {0.04 - 0.01}) node[left, blue] {$f_2$}
        (1,-0.08) node[right, blue] {$f_3$}
        (1, 0.06) node[right, blue] {$f_4$}
        (1, 0.23) node[right, darkgray] {$f$}
    ;
    \draw[darkgray] 
        (0, 0.25) parabola bend (3/4, 0.17) (1, 0.23)
    ;
    \end{tikzpicture}
    \caption{A representation of the four lines $f_1,\dots,f_4$. 
    By convexity, $f$ is lower bounded by the blue solid segments.
    Note that, since $J_{l_{\tau_T},T} \cap J_{c_{\tau_T},T} \cap J_{r_{\tau_T},T} \neq \varnothing$, then $f_1,f_2$ are nondecreasing and $f_3,f_4$ nonincreasing. 
    Therefore, the minimum of the $y$ coordinates of the red dots is a lower bound on the value of the function, which in turn implies \eqref{eq:straightlines}.}
    \label{f:lines}
\end{figure}
By the convexity of $f$, for each $x \in [I_{\tau_T}^-,l_{\tau_T}]$ we have $f(x) \ge f_1(x)$, for each $x \in [l_{\tau_T},c_{\tau_T}]$ we have $f(x) \ge f_2(x)$, for each $x \in [c_{\tau_T},r_{\tau_T}]$ we have $f(x) \ge f_3(x)$, and for each $x \in [r_{\tau_T},I_{\tau_T}^+]$ we have $f(x) \ge f_4(x)$. Writing down explicitly these four inequalities and upper bounding, we conclude that
\begin{align}
&
    f(x^\star)
\quad\ge
    \min_{x' \in \{l_{\tau_T}, c_{\tau_T}, r_{\tau_T}\}} J_{x,T}^- 
\nonumber
\\
&
    - 2 \lrb{ \max_{x' \in \{l_{\tau_T}, c_{\tau_T}, r_{\tau_T}\}} J_{x',T}^+  - \min_{x' \in \{l_{\tau_T}, c_{\tau_T}, r_{\tau_T}\}} J_{x',T}^- } \;.
\label{eq:straightlines}
\end{align}
Then, if $B_{\tau_T} \ge 4 \max_{t \in [T]} b_t$, for all $x \in \{l_{\tau_T}, c_{\tau_T}, r_{\tau_T}\}$, we have
\begin{align*}
&
    f(x) - f(x^\star)
\le
    J_{x,T}^+ - f(x^\star)
\\
\overset{\eqref{eq:straightlines}}&{\le}
    J_{x,T}^+ - \min_{x' \in \{l_{\tau_T}, c_{\tau_T}, r_{\tau_T}\}} J_{x,T}^- 
\\
&
\quad \quad
    + 2 \lrb{ \max_{x' \in \{l_{\tau_T}, c_{\tau_T}, r_{\tau_T}\}} J_{x',T}^+  - \min_{x' \in \{l_{\tau_T}, c_{\tau_T}, r_{\tau_T}\}} J_{x',T}^- } 
\\
\overset{\eqref{eq:skewer}}&{\le}
    6 \max_{x' \in \{l_{\tau_T}, c_{\tau_T}, r_{\tau_T}\}} \labs{ J_{x',T} }
\le
    6 \max_{x' \in \{l_{\tau_T}, c_{\tau_T}, r_{\tau_T}\}} \frac{c}{\fB_{x',T}^\alpha}
\\
\overset{\eqref{eq:equipartition}}&{\le}
    6 \frac{c}{ \lrb{ \frac{B_{\tau_T} - 2 \max_{t \in [T]} b_t  } {3} }^\alpha }
\le
    6 \cdot 3^\alpha \frac{c}{ (B_{\tau_T}/2)^\alpha }
\\
&
\le
    6 \cdot 12^\alpha \cdot \frac{c}{( \sum_{t=1}^T b_t )^\alpha}
\le
    12 \cdot\lrb{ \frac{48}{2^{1/\alpha} - 1} } ^\alpha \cdot \frac{c}{( \sum_{t=1}^T b_t )^\alpha}
    \;.
\end{align*}
If, on the other hand, $B_{\tau_T} < 4 \max_{t \in [T]} b_t$, since $B_{\tau_T} \ge \frac12 \sum_{t=1}^T b_t$, for all $x \in \{l_{\tau_T}, c_{\tau_T}, r_{\tau_T}\}$, we have
\begin{align*}
&
    f(x) - f(x^\star)
\le
    L \labs I (\nicefrac34)
\le
    L \labs I \lrb{\nicefrac12}^{\frac{4 \max_{t \in [T]} b_t}{12 \max_{t \in [T]} b_t}}
\\
&
\qquad
\le
    L \labs I \lrb{\nicefrac12}^{\frac{B_{\tau_T}}{12\max_{t \in [T]} b_t}}
\le
    L \labs I \lrb{\nicefrac12}^{\frac{\sum_{t=1}^T b_t}{24\max_{t \in [T]} b_t}}
\\
&
\qquad
=
     L \labs I \exp \lrb{ -\frac{\ln 2}{24}  \frac{{\sum_{t=1}^T b_t}}{\max_{t \in [T]} b_t } }
\end{align*}
where $L$ the \lip{} constant of $f$ on $[l_{\tau_T}, r_{\tau_T}]$.
Thus, adding together the two bounds for $B_{\tau_T} \ge 4 \max_{t \in [T]} b_t$ and $B_{\tau_T} < 4 \max_{t \in [T]} b_t$ yields  \eqref{e:upper-bound} when $\del_T = \cancNone$ and $B_{\tau_T} \ge \sum_{\tau' = 0}^{\tau_T -1} B_{\tau'}$.

Finally, assume that $\del_T = \cancNone$ and $B_{\tau_T} < \sum_{\tau' = 0}^{\tau_T -1} B_{\tau'}$ (i.e., the condition on \Cref{state:recommendation-two} is true and we recommend $R_T$ as in \Cref{state:recommendationtwo}).
Then the recommendation $R_T$ is the point with the best upper bound at the end of the second-to-last epoch.
Thus, proceeding as in the first part of the proof (case $\del_T \neq \cancNone$), we get
\begin{align*}
&
    f(R_T) - f(x^\star)
=
    f\brb{ R_{t_{\tau_T-1}} } - f(x^\star)
\\
&
\qquad
\le
    12 \lrb{ \frac{24}{2^{1/\alpha} - 1} } ^\alpha  \frac{c }{ (\sum_{t=1}^{t_{\tau_T-1}} b_t)^\alpha } 
\\
&
\qquad\qquad
    + \frac{9}{8} L \labs{I} \exp \lrb{- \frac{\ln 2}{24} \frac{\sum_{t=1}^{t_{\tau_T-1}} b_t}{\max_{t \in [t_{\tau_T-1}]} b_t}} 
\\
&
\qquad
<
    12 \lrb{ \frac{48}{2^{1/\alpha} - 1} } ^\alpha \frac{c}{(\sum_{t=1}^T b_t)^{\alpha}}
\\
&
\qquad\qquad
    + \frac{9}{8} L \labs{I} \exp \lrb{- \frac{\ln 2}{48} \frac{\sum_{t=1}^T b_t}{\max_{t \in [T]} b_t}}
\end{align*}
where $L$ the \lip{} constant of $f$ on $[l_{\tau_T}, r_{\tau_T}]$. Being the interval $I$, the time $T$ and the convex function $f$ arbitrarily chosen, the proof is complete.
\end{proof}

\end{document}